\documentclass[11pt,onecolumn]{amsart}

\usepackage[margin=1.15in]{geometry}
\usepackage[T1]{fontenc}
\usepackage{lmodern}
\usepackage{microtype}
\usepackage{amsmath}
\usepackage{amssymb}
\usepackage{amsthm}
\usepackage{aliascnt}
\usepackage{mathtools}
\usepackage{fancyhdr}
\usepackage{hyperref}
\usepackage{cleveref}
\usepackage[shortlabels]{enumitem}

\usepackage{tikz}
\usetikzlibrary{calc,arrows.meta}

\setlist[itemize]{leftmargin=2.2em,itemsep=0.25ex,topsep=0.6ex}
\setlist[enumerate]{leftmargin=2.2em,itemsep=0.25ex,topsep=0.6ex}
\theoremstyle{plain}
\newtheorem{theorem}{Theorem}[section]
\newtheorem{proposition}[theorem]{Proposition}
\newtheorem{lemma}[theorem]{Lemma}
\newtheorem{corollary}[theorem]{Corollary}

\theoremstyle{definition}
\newtheorem{definition}[theorem]{Definition}

\theoremstyle{remark}
\newtheorem{remark}[theorem]{Remark}

\newcommand{\RowNorm}{\operatorname{RowNorm}}
\newcommand{\T}{\mathcal T}
\newcommand{\V}{\mathcal V}
\newcommand{\NH}{\mathcal N}

\newcommand{\disc}{\operatorname{disc}}

\title{Bang--bang representation of $3\times 3$ embeddable stochastic matrices}

\author{Leonel Robert}
\address{\parbox{\linewidth}{Department of Mathematics, University of Louisiana at Lafayette, \\
		217 Maxim Doucet Hall, 1401 Johnston Street, Lafayette, LA 70503, USA}}
\email{lrobert@louisiana.edu}

\begin{document}
\begin{abstract}
	We prove that every embeddable $3\times3$ row-stochastic matrix is a
	product of at most seven elementary row-stochastic matrices. Frydman's
	example shows that this bound is sharp. The proof combines Frydman's
	six-factor criterion and spiral theorem with a reduction to a
	three-parameter critical family and an analysis of the membership
	certificates that persist along it.
\end{abstract}

\maketitle
\section{Introduction}
In this note we present a solution to the bang--bang problem for
non-homogeneous continuous-time Markov processes with three states. This problem
was actively investigated in the 1970s and early 1980s by Goodman,
Johansen--Ramsey, and Frydman, among others
\cite{Goodman1970,JohansenRamsey1979,FrydmanSinger1979,FrydmanEmbedding1980,Frydman1980,Frydman1983}.
Since then, it has received almost no attention (but see \cite{BaakeSumner2024,Singer2025}).

By a non-homogeneous continuous-time Markov process with $n$ states we mean a continuous map $(s,t)\mapsto P(s,t)$, for $0\leq s\leq t<\infty$, 
where $P(s,t)$ is a row-stochastic $n\times n$ matrix, such that $P(s,s)=\mathrm{Id}$ for all $s\geq 0$ and
\[
P(r,s)P(s,t)=P(r,t)\text{ for all }r\leq s\leq t.
\]
 An $n\times n$ row-stochastic matrix $P$ is called embeddable, or attainable, if  $P=P(0,1)$ for some  Markov process $P(\cdot,\cdot)$.

For indices $1\leq i,j\leq n$ with $i\neq j$, and for $0\leq t<1$, put
\[
K_{ij}(t)=I+t(E_{ij}-E_{ii}).
\]
 We call these matrices \emph{elementary row-stochastic matrices}, abbreviated e.r.s. matrices. 
Left multiplication of $P$ by $K_{ij}(t)$ replaces row $i$ by
$(1-t)P_i+tP_j$ (where $P_i,P_j$ are the rows $i$ and $j$ of $P$) and leaves the other rows unchanged. 

Every e.r.s. matrix is attainable, as is any finite product of e.r.s. matrices.
Conversely, by the chattering principle, every attainable matrix is a limit of finite products of e.r.s.
matrices (see \cite{JohansenCLT, Sussmann1972}). For $n=3$, Johansen and
Ramsey showed in \cite{JohansenRamsey1979} that every attainable $P$ is a finite
product of e.r.s. matrices,
and that if $P$ is close enough to the identity, then it is a product of six e.r.s. matrices. On the other hand,
Frydman gave in \cite{Frydman1983} an example of an attainable $3\times 3$ row-stochastic matrix expressible as a product of seven
but no fewer e.r.s. matrices. Here we show the following:

\begin{theorem}
\label{thm:main}
Every attainable $3\times 3$ row-stochastic matrix is expressible as a
product of at most seven e.r.s. matrices.
\end{theorem}

This implies that the set of attainable matrices is semialgebraic, by the  Tarski--Seidenberg theorem.
An explicit description of this set can be obtained from Frydman's description of
the products of six e.r.s. matrices in \cite{Frydman1980}.
 
The problem of describing the set of attainable $3\times 3$ row-stochastic matrices has a geometric formulation due to Goodman \cite{GoodmanCIME}. Fix a
nondegenerate triangle $ABC$ in the plane. We consider decreasing paths of triangles $t\mapsto A_tB_tC_t$ that start at $ABC$. Here, by decreasing we mean that $A_{t}B_tC_t$ is contained in $A_sB_sC_s$ whenever $s\leq t$. 

\begin{figure}[ht]
	\centering
	\begin{tikzpicture}[scale=1, every node/.style={font=\scriptsize}]
		
		\coordinate (A) at (0,0);
		\coordinate (B) at (5,0);
		\coordinate (C) at (1.6,3.2);
		
		\coordinate (O) at ($0.333*(A)+0.333*(B)+0.333*(C)$);
		
		\def\rOne{0.84}
		\def\rTwo{0.68}
		\def\rFinal{0.50}
		
		\coordinate (Aone) at ($(O)!\rOne!(A)$);
		\coordinate (Bone) at ($(O)!\rOne!(B)$);
		\coordinate (Cone) at ($(O)!\rOne!(C)$);
		
		\coordinate (Atwo) at ($(O)!\rTwo!(A)$);
		\coordinate (Btwo) at ($(O)!\rTwo!(B)$);
		\coordinate (Ctwo) at ($(O)!\rTwo!(C)$);
		
		\coordinate (Ap) at ($(O)!\rFinal!(A)$);
		\coordinate (Bp) at ($(O)!\rFinal!(B)$);
		\coordinate (Cp) at ($(O)!\rFinal!(C)$);
		
		\draw[thick] (A)--(B)--(C)--cycle;
		
		\draw[dashed] (A)--(Ap);
		\draw[dashed] (B)--(Bp);
		\draw[dashed] (C)--(Cp);
		
		\draw[dashed, gray] (Aone)--(Bone)--(Cone)--cycle;
		\draw[dashed, gray] (Atwo)--(Btwo)--(Ctwo)--cycle;
		
		\draw[thick] (Ap)--(Bp)--(Cp)--cycle;
		
		\fill (A) circle (1.4pt) node[below left] {$A$};
		\fill (B) circle (1.4pt) node[below right] {$B$};
		\fill (C) circle (1.4pt) node[above] {$C$};
		
		\fill (Ap) circle (1.4pt) node[below left] {$A'$};
		\fill (Bp) circle (1.4pt) node[below right] {$B'$};
		\fill (Cp) circle (1.4pt) node[above] {$C'$};
		
	\end{tikzpicture}
	\caption{A decreasing path from $ABC$ to $A'B'C'$.}
	\label{fig:decreasing-path}
\end{figure}
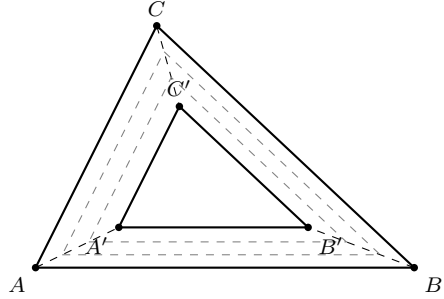

A nondegenerate triangle $A'B'C'$ contained in $ABC$ is called attainable if there is a decreasing
path from $ABC$ to $A'B'C'$. It is not difficult to show that $A'B'C'$ is attainable if and only if the matrix $P_{A'B'C'}$ whose rows are the barycentric coordinates of $A',B',C'$ with respect to $A,B,C$ is an attainable $3\times 3$ row-stochastic matrix \cite{GoodmanCIME}. Left multiplication of $P_{A'B'C'}$  by $K_{ij}(t)$ applies a pull-in move on $A'B'C'$, i.e., one of its vertices is moved towards one of the two other vertices along the segment connecting them. 

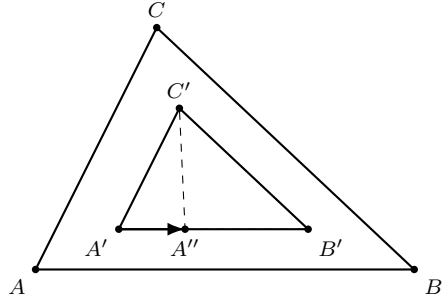
\begin{figure}[ht]
	\centering
	\begin{tikzpicture}[scale=1, every node/.style={font=\scriptsize}]
		
		\coordinate (A) at (0,0);
		\coordinate (B) at (5,0);
		\coordinate (C) at (1.6,3.2);
		
		\coordinate (O) at ($0.333*(A)+0.333*(B)+0.333*(C)$);
		\def\rFinal{0.5}
		
		\coordinate (Ap) at ($(O)!\rFinal!(A)$);
		\coordinate (Bp) at ($(O)!\rFinal!(B)$);
		\coordinate (Cp) at ($(O)!\rFinal!(C)$);
		
		\coordinate (App) at ($(Ap)!0.35!(Bp)$);
		
		\draw[thick] (A)--(B)--(C)--cycle;
		
		\draw[thick] (Ap)--(Bp)--(Cp)--cycle;
		
		\draw[dashed] (App)--(Bp)--(Cp)--cycle;
		
		\draw[-Latex,thick] (Ap)--(App);
		
		\fill (A) circle (1.4pt) node[below left] {$A$};
		\fill (B) circle (1.4pt) node[below right] {$B$};
		\fill (C) circle (1.4pt) node[above] {$C$};
		
		\fill (Ap) circle (1.4pt) node[below left] {$A'$};
		\fill (Bp) circle (1.4pt) node[below right] {$B'$};
		\fill (Cp) circle (1.4pt) node[above] {$C'$};
		\fill (App) circle (1.4pt) node[below] {$A''$};
		
	\end{tikzpicture}
	\caption{A pull-in move on $A'B'C'$: the vertex $A'$ is moved toward
		$B'$ along the segment $A'B'$.}
	\label{fig:pull-in}
\end{figure}

Theorem \ref{thm:main} has the following corollary:

\begin{corollary}
If a nondegenerate triangle $A'B'C'$ is attainable by a decreasing path from a triangle $ABC$,  then $A'B'C'$ can be attained
in at most seven pull-in moves starting from $ABC$.	
\end{corollary}

We will not make explicit use of this geometric point of view
here, but it informs many of the arguments and calculations below. The attainability problem for polygons in the plane with more than three vertices
was solved in \cite{DecreasingPaths}.

\section{Preliminaries}
\label{sec:preliminaries}

All matrices in this section are $3\times3$ matrices. Given a matrix $P$, we 
write $P>0$ to indicate that all the entries of $P$
are positive.


\subsection{The stochastic formulation}
Let
\[
R_n=\left\{
K_{i_nj_n}(s_n)\cdots K_{i_1j_1}(s_1):
 i_\ell\neq j_\ell,\ 0\leq s_\ell<1
\right\},
\]
and put $R_\infty=\bigcup_nR_n$.

By a pull-in move on the set of $3\times 3$ row-stochastic matrices we mean a path of the form
\[
[0,c]\ni s\mapsto K_{ij}(s)P,
\]
where $0\leq c<1$. A pull-in path  is  a finite concatenation of pull-in moves, where left multiplication by $K_{ij}(s)$
appends the last move. Note that each  $P\in R_n$ can be reached by  a pull-in path starting at the identity matrix and
obtained by the concatenation of $n$ pull-in moves.

\begin{definition}
The threshold region $\T$ consists of the matrices in $R_\infty$ with at least one
zero entry.  

The vestibule $\V$ consists of the matrices $P>0$
for which $P=QK_{ij}(s)$ for some $Q\in\T$, some $i\neq j$, and some $0 < s <1$. 

The narrow
hall $\NH$ consists of the matrices $P\notin\T\cup\V$ for which
$P=QK_{ij}(s)$ for some $Q\in\V$, some $i\neq j$, and some $0<s<1$.
\end{definition}


Let $P=(p_{ij})_{ij}$ and write $D=\det P$.  If $M_{ij}$ is the unsigned
minor of $P$ obtained by deleting row $i$ and column $j$, define the modified minors
by $T_{ii}=M_{ii}$ and $T_{ij}=(-1)^{i+j-1}M_{ij}$ for $i\neq j$.  Thus,
for a permutation $(i,j,k)$ of $(1,2,3)$,
\[
T_{ij}=p_{ji}p_{kk}-p_{jk}p_{ki},
\qquad
T_{ii}=p_{jj}p_{kk}-p_{jk}p_{kj}.
\]
We frequently use  that
\[
\det(P)=p_{jj}T_{jj}-p_{ji}T_{ji}-p_{jk}T_{jk}
\]
for any permutation $(i,j,k)$.

Set
\[
G(P)=p_{11}p_{22}p_{33}-D.
\]
Also put
\begin{align*}
V_{ii}(P) &=p_{ii}T_{ii}-D,\\
V_{ij}(P) &=p_{ii}p_{jj}T_{ij}-p_{ji}D
\quad(i\neq j).
\end{align*}

For $c\geq0$, let
\[
Z_{ij}(c)=I+cE_{ii}-cE_{ij}.
\]
If
$s=c/(1+c)$, then $Z_{ij}(c)=K_{ij}(s)^{-1}$.  The matrix
$PZ_{ij}(c)$ is nonnegative precisely for
$0\leq c\leq\bar c_{ij}$, where
\[
\bar c_{ij}=\min_{1\leq\ell\leq3}\frac{p_{\ell j}}{p_{\ell i}}.
\]

We define Frydman's quadratic for the ordered pair $(i,j)$ as
\[
q_{ij}^P(c)=\alpha_{ij}c^2+\beta_{ij}c+\gamma_{ij}
\]
where
\begin{equation}\label{frydmancoeffs}
\alpha_{ij}=p_{ji}p_{ii}T_{jj},
\qquad
\beta_{ij}=p_{ii}(D-p_{jj}T_{jj}-p_{ji}T_{ji}),
\qquad
\gamma_{ij}=V_{ji}(P).
\end{equation}
A direct calculation gives
$V_{ji}(PZ_{ij}(c))=(1+c)q_{ij}^P(c)$.

\begin{theorem}[Frydman's $R_6$ test]
\label{thm:Frydman-finite-move}
Let $P$ be a $3\times 3$ row-stochastic matrix with $D>0$.
\begin{enumerate}[(i)]
\item If $P$ has a zero entry, 
\[
P\in\T\Longleftrightarrow G(P)\geq 0.
\]
In that case $P\in R_5$, and $G(P)=0$ implies $P\in R_4$.

\item If $P>0$, then $P\in\V$ if and only if
$\max_{i,j}V_{ij}(P)\geq0$.  More generally,
\[
P\in\T\cup\V
\quad\Longleftrightarrow\quad
\max_{i,j}V_{ij}(P)\geq0.
\]

\item If $P\notin\T\cup\V$, then $P\in\NH$ if and only if, for some
ordered pair $(i,j)$, the quadratic $q_{ij}^P$ has a root in
$[0,\bar c_{ij}]$.

\item One has 
\[
R_6=\T\cup\V\cup\NH.
\]
\end{enumerate}
\end{theorem}

Part (i) is  \cite[Theorem 3.1]{FrydmanEmbedding1980}
and \cite[Lemma~2.1]{Frydman1980}.
Parts (ii) and (iii) are the first two steps in the proof of
\cite[Theorem~2.1]{Frydman1980}; the quadratic in (iii) is
\cite[Equation~(2.2)]{Frydman1980}. Part (iv) is
\cite[Theorem~2.1]{Frydman1980}, restated as
\cite[Theorem~1.1]{Frydman1983}. Our $V_{ij}$ polynomials are denominator-free
versions of Frydman's quantities $B(i,j)-\det P$.

\subsection{The homogeneous formulation}

Let
\[
\mathcal M_+=\{A\in M_3(\mathbb R):A\geq0,\ \det A>0\}
\]
and let $\mathcal D_+$ be the group of positive diagonal matrices.  Given $A,B\in \mathcal M_+$, we
write $A\sim B$ if $B=LAR$ for some $L,R\in\mathcal D_+$.

For $i\neq j$ and $t\geq0$, put
\[
X_{ij}(t)=I+tE_{ij}.
\]

We now extend the definitions of the sets $R_n$, $\T$, $\V$, and $\NH$ from the row-stochastic matrices to  $\mathcal M_+$, retaining the same notation. This causes no ambiguity, since the intersections of the enlarged sets with the row-stochastic matrices are precisely the sets defined above.

\begin{definition}
	The set $R_n\subseteq\mathcal M_+$ consists of the matrices $A$ for which
	\[
	A\sim
	X_{i_nj_n}(t_n)\cdots X_{i_1j_1}(t_1)
	\]
	for some $i_\ell\neq j_\ell$ and some $t_\ell\geq0$. Put
	$R_\infty=\bigcup_nR_n$.
	
	The threshold region $\T$ consists of matrices $A\in R_\infty$ with at least one zero entry.

	The vestibule $\V$  consists of the matrices $A>0$ such that $A\sim B X_{ij}(t)$
	for some $B\in\T$, some $i\neq j$, and some $t>0$.
	
	The narrow hall $\NH$ consists of the matrices $A\notin\T\cup\V$ such that $A\sim B X_{ij}(t)$
	for some $B\in\V$, some $i\neq j$, and some $t>0$.
\end{definition}

If $D=\operatorname{diag}(d_1,d_2,d_3)$, then
\[
	DX_{ij}(t)D^{-1}=X_{ij}\left(\frac{d_i}{d_j}t\right).
\]
This allows 
positive diagonal factors to be moved through a finite product of elementary matrices $X_{ij}$.

For $A\in\mathcal M_+$, let $\RowNorm(A)$ be obtained by dividing each
row of $A$ by the row sum.  Let $D_i(\lambda)$ denote the diagonal matrix whose $i$-th diagonal
entry is $\lambda$ and whose other diagonal entries are $1$. Then
\[
K_{ij}(s)=D_i(1-s)X_{ij}\left(\frac{s}{1-s}\right),
\qquad
X_{ij}(t)=D_i(1+t)K_{ij}\left(\frac{t}{1+t}\right).
\]
We readily deduce from these identities that  $A\in\mathcal M_+$ belongs to
$R_n,\T,\V,$ and $\NH$ if and only if $\RowNorm(A)$ belongs to the
corresponding stochastic region.

\begin{proposition}
\label{prop:Rn-relatively-closed}
For every $n$, the set $R_n$ is relatively closed in $\mathcal M_+$.
\end{proposition}

\begin{proof}
We first consider row-stochastic matrices. Fix a sequence of ordered
pairs
\[
(i_1,j_1),\ldots,(i_n,j_n),
\qquad i_k\neq j_k,
\]
and consider the map
\[
[0,1]^n\longrightarrow M_3(\mathbb R),
\qquad
(t_1,\ldots,t_n)\longmapsto
K_{i_nj_n}(t_n)\cdots K_{i_1j_1}(t_1).
\]
This map is continuous, so its image is compact. There are only
finitely many possible sequences of ordered pairs, and hence the union
of these images is compact, and therefore closed.

If one of the parameters $t_k$ is equal to $1$, then the corresponding
factor $K_{i_kj_k}(t_k)$ has determinant zero, and so does the product.
Consequently, after intersecting this finite union with $\mathcal M_+$,
we obtain precisely the row-stochastic matrices in $R_n$. Thus the
row-stochastic part of $R_n$ is relatively closed in $\mathcal M_+$.

Finally, the homogeneous set $R_n$ is the inverse image of its
row-stochastic part under the continuous row-normalization map
$\RowNorm\colon\mathcal M_+\to\mathcal M_+$. Hence $R_n$ is relatively
closed in $\mathcal M_+$.
\end{proof}

The polynomials $G,T_{ij},V_{ij}$ are defined on $\mathcal M_+$ by the
same formulas given above.  For $A=(a_{ij})_{ij}>0$ and $i\neq j$, put
\[
C_{ij}(A)=\min_{1\leq\ell\leq3}
\frac{a_{\ell j}}{a_{\ell i}},
\]
and define
\[
q_{ij}^A(c)=\alpha_{ij}c^2+\beta_{ij}c+\gamma_{ij},
\]
where
\[
\alpha_{ij}=a_{ji}a_{ii}T_{jj},
\qquad
\beta_{ij}
=a_{ii}\bigl(\det A-a_{jj}T_{jj}-a_{ji}T_{ji}\bigr),
\qquad
\gamma_{ij}=V_{ji}(A).
\]
Thus $AZ_{ij}(c)$ is nonnegative precisely when
\[
0\leq c\leq C_{ij}(A),
\]
and
\[
V_{ji}\bigl(AZ_{ij}(c)\bigr)=(1+c)q_{ij}^A(c).
\]

We have the following:

\begin{lemma}
\label{lem:diagonal-transpose-covariance}
Let $L=\operatorname{diag}(\lambda_1,\lambda_2,\lambda_3)$ and
$R=\operatorname{diag}(\rho_1,\rho_2,\rho_3)$ be positive, and put
$\lambda=\det L$ and $\rho=\det R$. Let $A\in \mathcal M_+$ and put $B=LAR$. Then
\[
T_{ij}(B)=\frac{\lambda\rho}{\lambda_i\rho_j}T_{ij}(A),
\qquad
G(B)=\lambda\rho G(A),
\]
\[
V_{ii}(B)=\lambda\rho V_{ii}(A),
\qquad
V_{ij}(B)=\lambda\rho\lambda_j\rho_iV_{ij}(A)
\quad(i\neq j),
\]
\[
C_{ij}(B)=\frac{\rho_j}{\rho_i}C_{ij}(A),
\]
and
\[
q_{ij}^{B}(c)
=\lambda\rho\lambda_i\rho_j\,
 q_{ij}^A\!\left(\frac{\rho_i}{\rho_j}c\right).
\]
Moreover,
$T_{ij}(A^t)=T_{ji}(A)$ and
$V_{ij}(A^t)=V_{ji}(A)$.
\end{lemma}

\begin{proof}
The formula for $T_{ij}$ follows by deleting row $i$ and column $j$.
The formulas for $G$ and $V_{ij}$ follow by substitution. The formula
for $q_{ij}$ follows from its three coefficients. Transposition
interchanges the corresponding minors.
\end{proof}

\begin{theorem}
	\label{thm:homogeneous-finite-move}
	Let $A\in\mathcal M_+$.
	\begin{enumerate}[(i)]
		\item
		If $A$ has a zero entry, then 
		\[
		A\in\T\Longleftrightarrow G(A)\geq0.
		\]
		In that case $A\in R_5$, and $G(A)=0$ implies $A\in R_4$.
		
		\item
		If $A>0$, then $A\in\V$ if and only if $\max_{i,j}V_{ij}(A)\geq0$.
		More generally, if $A\in \mathcal M_+$, then
		\[
		A\in\T\cup\V
		\quad\Longleftrightarrow\quad
		\max_{i,j}V_{ij}(A)\geq0.
		\]
		
		\item
		If $A\notin\T\cup\V$, then $A\in\NH$ if and only if, for some
		ordered pair $(i,j)$ with $i\neq j$, the quadratic $q_{ij}^A$
		has a root in $[0,C_{ij}(A)]$.

		\item
		One has
		\[
		R_6=\T\cup\V\cup\NH.
		\]
	\end{enumerate}
\end{theorem}

\begin{proof}
	Let
	$P=\RowNorm(A)$.
	Then $P$ is row-stochastic, $A\sim P$, and
	$A$ belongs to any one of the regions $R_n,\T,\V,$ and $\NH$ if and
	only if $P$ belongs to the corresponding stochastic region.
	
	By Lemma~\ref{lem:diagonal-transpose-covariance}, diagonal equivalence
	preserves the signs of $G$ and of all the $V_{ij}$. Moreover, row
	normalization does not change the quotients occurring in $C_{ij}$, so
	\[
	C_{ij}(P)=C_{ij}(A).
	\]
	The same lemma shows that $q_{ij}^P$ is a positive scalar multiple of
	$q_{ij}^A$. Thus the two quadratics have the same roots. The result now
	follows from Theorem~\ref{thm:Frydman-finite-move}.
\end{proof}

\begin{proposition}
\label{prop:first-last-symmetry}
Each of the regions $R_n$, $\T$, $\V$, and $\NH$ is invariant under
transposition. For $A>0$, one has $A\in\V$ if and only if, for some
$B\in\T$, $i\neq j$, and $t>0$,
\[
A\sim BX_{ij}(t),
\]
and this is also equivalent to
\[
A\sim X_{ij}(t)B.
\]
Likewise, $A\in\NH$ if and only if $A\notin\T\cup\V$ and, for some
$B\in\V$, $i\neq j$, and $t>0$,
\[
A\sim BX_{ij}(t),
\]
equivalently, $A\sim X_{ij}(t)B$.
\end{proposition}

\begin{proof}
Since $X_{ij}(t)^t=X_{ji}(t)$, transposition preserves $R_n$
and $\T$.  Lemma~\ref{lem:diagonal-transpose-covariance} and
Theorem \ref{thm:homogeneous-finite-move} (ii) give the invariance of $\V$;
then $\NH=R_6\setminus(\T\cup\V)$ is invariant.  Transposing the
first-move descriptions of $\V$ and $\NH$ gives their last-move descriptions.
\end{proof}

Thus the right inverse $AZ_{ij}(c)$ undoes a first move, while
$X_{ij}(-t)A$ undoes a last move.  We use either description as
convenient. If $P$ is row-stochastic, then $\RowNorm(P^t)$ represents the
transpose class.  Hence, the row-stochastic regions $R_n,\T,\V,\NH$ are invariant under
$P\mapsto\RowNorm(P^t)$.

\begin{corollary}
\label{cor:no-exit-vestibule}
If $A\in\T\cup \V$, then $X_{ij}(t)A\in R_6$ for every $i\neq j$ and
$t\geq0$.  Consequently, a first exit from $R_6$ along a pull-in path
must occur through $\NH$.
\end{corollary}

\begin{proof}
If $A\in \T$, then 	
the matrix $X_{ij}(t)A$ is either in $\T$ or in $\V$, by the last-move description of $\V$. 
If $A\in \V$, then $X_{ij}(t)A$ is either in $\V$ or in $\NH$, by the last-move
description of $\NH$.  For the last assertion, apply
Theorem \ref{thm:homogeneous-finite-move} (iv).  
\end{proof}

\begin{remark}
	Let us return to the geometric picture of attainable triangles and describe
	the regions $\T$, $\V$, and $\NH$ in these terms.
	A nondegenerate triangle
	$A'B'C'$ contained in $ABC$ lies in the threshold region
	(i.e., $P_{A'B'C'}\in\T$) if it is attainable and at least one of its
	vertices lies on an edge of $ABC$.
	The triangle $A'B'C'$ lies in the vestibule if it is contained strictly
	inside $ABC$ and can be obtained from a threshold triangle by a single
	pull-in move. Finally, $A'B'C'$ lies in the narrow hall if it does not
	belong to the vestibule and can be obtained from a vestibule triangle
	by a single pull-in move.
\end{remark}

\subsection{Frydman's discriminant and the reduced quadratic test}

We now analyze the discriminant of the Frydman quadratic.  For
$A>0$ and $i\neq j$, put
\[
H_{ij}(A)=\operatorname{disc}(q_{ij}^A).
\]
When the matrix $A$ is understood, we write simply $H_{ij}$ and
$C_{ij}$.

\begin{lemma}
\label{lem:Frydman-discriminant-factorization}
For every permutation $(i,j,k)$ of $(1,2,3)$,
\[
H_{ij}=a_{ii}T_{jk}\widetilde H_{ij}
\]
where
\[
\widetilde H_{ij}=4a_{ji}a_{ik}D+a_{ii}a_{jk}^2T_{jk}.
\]
\end{lemma}

\begin{proof}
Substitution of the formulas for $\alpha_{ij}$, $\beta_{ij}$, and
$\gamma_{ij}$, followed by simplification, gives the stated
factorization.
\end{proof}

\begin{proposition}[Reduced Frydman test]
\label{prop:reduced-Frydman-test}
Assume that $A>0$ and $A\notin\V$.  Fix a permutation $(i,j,k)$ and
write $C=C_{ij}$.  Then $q_{ij}^A$ has a root in $[0,C]$ if and only if
\[
T_{jj}<0,
\qquad
\beta_{ij}>0,
\qquad
2\alpha_{ij}C+\beta_{ij}<0,
\qquad
T_{jk}<0,\qquad \widetilde H_{ij}\leq0.
\]
\end{proposition}

\begin{proof}
Theorem \ref{thm:homogeneous-finite-move} (ii) gives
$q_{ij}^A(0)=V_{ji}(A)<0$.  At $c=C$, the matrix $B=AZ_{ij}(C)$
is nonnegative and has a zero entry. If $q_{ij}^A(C)\geq0$, then \(V_{ji}(B)\geq0\), 
so Theorem \ref{thm:homogeneous-finite-move} (ii) places $B$ in $\T$.  The first-move description of $\V$ would then place $A$ in
$\V$, a contradiction.  Thus both endpoint values are negative.

A quadratic negative at both endpoints has a zero in the interval
exactly when it is concave, its vertex lies in $(0,C)$, and its maximum
is nonnegative.  This is equivalent to
\[
T_{jj}<0,
\qquad
\beta_{ij}>0,
\qquad
2\alpha_{ij}C+\beta_{ij}<0,
\qquad H_{ij}\geq0.
\]
We show next that these inequalities are equivalent to the five displayed inequalities
in the statement. The factorization of $H_{ij}$ in Lemma~\ref{lem:Frydman-discriminant-factorization}
shows that $T_{jk}<0$ and $\widetilde H_{ij}\leq0$ implies $H_{ij}\geq 0$, which gives one direction.

Let us prove the other direction. Expanding
$D$ along row $j$ in the formula for $\beta_{ij}$ gives
\[
\beta_{ij}=-a_{ii}(2a_{ji}T_{ji}+a_{jk}T_{jk}).
\]
Hence,
\[
-\frac{\beta_{ij}}{2\alpha_{ij}}
=\frac{2a_{ji}T_{ji}+a_{jk}T_{jk}}{2a_{ji}T_{jj}}\leq C\leq \frac{a_{jj}}{a_{ji}}.
\]
Since $T_{jj}<0$, the preceding inequality gives
\[
2a_{ji}T_{ji}+a_{jk}T_{jk}
\geq 2a_{jj}T_{jj}.
\]
Using
\[
D=a_{jj}T_{jj}-a_{ji}T_{ji}-a_{jk}T_{jk},
\]
this is equivalent to
\[
-a_{jk}T_{jk}\geq2D.
\]
Hence $T_{jk}<0$. The inequality $\widetilde H_{ij}\leq0$ now follows from Lemma~\ref{lem:Frydman-discriminant-factorization}.
\end{proof}

\subsection{The last-move quadratic}

Let \(A=(a_{ij})_{ij}>0\), and fix \(i\neq j\). To undo a last move of
type \(X_{ij}\), consider
\[
A_{ij}(t)=X_{ij}(-t)A.
\]
This matrix is nonnegative precisely for
\[
0\leq t\leq S_{ij},
\qquad
S_{ij}=\min_{1\leq\ell\leq3}\frac{a_{i\ell}}{a_{j\ell}}.
\]

Define the last-move quadratic
\[
G_{ij}^A(t)=V_{ji}\bigl(A_{ij}(t)\bigr).
\]
A direct calculation gives
\[
\begin{aligned}
	G_{ij}^A(t)
	={}&a_{jj}(a_{ii}-ta_{ji})(T_{ji}-tT_{ii})
	-D(a_{ij}-ta_{jj})\\
	={}&\alpha_{ij}^Gt^2+\beta_{ij}^Gt+\gamma_{ij}^G,
\end{aligned}
\]
where
\[
\alpha_{ij}^G=a_{jj}a_{ji}T_{ii},
\qquad
\beta_{ij}^G
=
a_{jj}\bigl(D-a_{ii}T_{ii}-a_{ji}T_{ji}\bigr),
\qquad
\gamma_{ij}^G=V_{ji}(A).
\]

\begin{proposition}
	\label{prop:last-move-quadratic-criterion}
	Suppose that \(A>0\) and \(A\notin\V\). Then \(A\in\NH\) if and only
	if, for some ordered pair \((i,j)\), the quadratic \(G_{ij}^A\) has a
	root in \([0,S_{ij}]\).
\end{proposition}

\begin{proof}
This is the last-move form of
Theorem \ref{thm:homogeneous-finite-move} (iii), obtained from the
first-move criterion by transposition.
\end{proof}

If \(A\notin\V\), then \(G_{ij}^A(0)=V_{ji}(A)<0\). Moreover,
\(G_{ij}^A(S_{ij})<0\): otherwise the endpoint
\(X_{ij}(-S_{ij})A\), which has a zero entry, would belong to
\(\T\), and the last-move description would imply \(A\in\V\).
Consequently, \(G_{ij}^A\) has a root in its admissible interval if
and only if
\[
\alpha_{ij}^G<0,
\qquad
\beta_{ij}^G>0,
\qquad
2\alpha_{ij}^GS_{ij}+\beta_{ij}^G<0,
\qquad
\disc(G_{ij}^A)\geq0.
\]

\begin{lemma}
	\label{lem:last-move-Frydman-discriminants}
	Let \((i,j,k)\) be a permutation of \((1,2,3)\). Then
	\[
	\disc(G_{ij}^A)=\disc(q_{jk}^A).
	\]
\end{lemma}

\begin{proof}
	This follows by expanding the discriminant of \(G_{ij}^A\) and using
	the factorization in
	Lemma~\ref{lem:Frydman-discriminant-factorization}.
\end{proof}

\subsection{Frydman's minimal spiral representations}
The following theorem is \cite[Theorem~3.1]{Frydman1980}.

\begin{theorem}
	\label{thm:Frydman-spiral}
	Let $P$ be a row-stochastic matrix. Suppose  that
	$P\in R_n\setminus\bigl(R_{n-1}\cup\V\bigr)$ and   $n\geq 6$.
	Then $P$ has a minimal representation
	\[
	P=(K_1K_2K_3K_4K_5)(K_6\cdots K_n),
	\]
	where 	$K_1,\ldots,K_n$ are e.r.s. matrices and
	\[
	P'=K_1K_2K_3K_4K_5\in\V
	\]
	satisfies exactly one vestibule equality. Moreover, in any such
	representation, if the unique vestibule equality is $V_{ji}(P')=0$ and $k$ is the remaining state, then the move types of the e.r.s. matrices 
	$K_1,\ldots,K_n$ are, successively,
	\[
	(k,i),\ (j,k),\ (i,j),\ (k,i),\ (j,k),\ (i,j),\ldots .
	\]
\end{theorem}

We will apply the preceding theorem in the homogeneous setting introduced before, and by adding new elementary factors on the left  side rather than on the right.

\begin{corollary}
	\label{cor:left-Frydman-spiral}
Let $A\in \mathcal M_+$.  Suppose that $A\in R_n\setminus\bigl(R_{n-1}\cup\V\bigr)$
 and $n\geq 6$.
	Then $A$ has a minimal representation
	\[
	A=(X_n\cdots X_6)(X_5\cdots X_1),
	\]
	where each $X_\ell$ is a homogeneous elementary matrix and
	\[
	A'=X_5\cdots X_1\in\V
	\]
	satisfies exactly one vestibule equality. Moreover, in any such
	representation, if the unique vestibule equality is
	$V_{ij}(A')=0$ and $k$ is the remaining state, then, read in the order in which the moves are applied, the move types of
	$X_1,\ldots,X_n$ are, successively,
	\[
	(i,k),\ (k,j),\ (j,i),\ (i,k),\ (k,j),\ (j,i),\ldots .
	\]
\end{corollary}

\begin{proof}
	Apply Theorem~\ref{thm:Frydman-spiral} after passing to a
	row-stochastic representative of $A^t$. Transposition reverses the
	order of the factors and sends a move of type $(a,b)$ to one of type
	$(b,a)$. It also sends the equality $V_{ji}=0$ to $V_{ij}=0$.
	The positive diagonal factors arising from row normalization do not
	change move types.
\end{proof}

\section{Reduction to the critical family}
\label{sec:critical-reduction}

We work in $\mathcal M_+$ with the homogeneous moves $X_{ij}$. In this section, we
reduce a hypothetical first exit from $R_7$ to a three-parameter
family of matrices.

Suppose that a pull-in path leaves $R_7$. Since $R_7$ is relatively closed
in $\mathcal M_+$, by Proposition~\ref{prop:Rn-relatively-closed}, there
is a first-exit point $A\in R_7$.
By relabeling the states, we
may assume that the pull-in path is generated by left multiplication by
$X_{23}$. Choose a sequence
$h_\nu\downarrow0$ such that
\[
A\in R_7,
\qquad
X_{23}(h_\nu)A\notin R_7.
\]
Then $A\notin R_6$, since otherwise
$X_{23}(h_\nu)A\in R_7$. Hence
\[
A\in R_7\setminus R_6.
\]

By Corollary~\ref{cor:left-Frydman-spiral}, choose a minimal
representation
\[
A=(X_7X_6)(X_5\cdots X_1)
\]
for which $A'=X_5\cdots X_1\in\V$ satisfies exactly one vestibule
equality. For every $\nu$, left multiplication by $X_{23}(h_\nu)$ gives an
eight-move representation of $X_{23}(h_\nu)A$. This representation is
minimal because $X_{23}(h_\nu)A\notin R_7$. It has the same
five-factor vestibule block $A'$, so the second part of
Corollary~\ref{cor:left-Frydman-spiral} applies.
Since the eighth move $X_{23}(h_\nu)$ has type $(2,3)$, the corresponding ordering
of the states is
\[
(i,j,k)=(1,3,2).
\]
Thus the move types of $X_1,\ldots,X_8$, read from right to left in the
displayed product, are
\[
(1,2),\ (2,3),\ (3,1),\ (1,2),\
(2,3),\ (3,1),\ (1,2),\ (2,3).
\]
Consequently,
\[
A=X_{12}(c_7)X_{31}(c_6)X_{23}(c_5)
X_{12}(c_4)X_{31}(c_3)X_{23}(c_2)X_{12}(c_1),
\]
where every $c_i>0$.

Let
\[
A_0=X_{31}(c_6)X_{23}(c_5)X_{12}(c_4)
    X_{31}(c_3)X_{23}(c_2)X_{12}(c_1)
\]
and
\[
A_1=X_{12}(c_7)X_{31}(c_6)X_{23}(c_5)
    X_{12}(c_4)X_{31}(c_3)X_{23}(c_2).
\]
These are the two consecutive six-move windows.

\begin{lemma}
\label{lem:two-R6-boundary-windows}
If $A$ is a first-exit point as above, then
$A_0,A_1\in\partial R_6$.
\end{lemma}

\begin{proof}
Both matrices lie in $R_6$.  If $A_0$ were interior to $R_6$, then
$X_{12}(-c_7)X_{23}(h)A$ would remain in $R_6$ for small $h$, and
multiplication by $X_{12}(c_7)$ would put $X_{23}(h)A$ in $R_7$, contradicting that
$X_{23}(h_\nu)A\notin R_7$ for $h_\nu\to 0$.  If
$A_1$ were interior to $R_6$, then $X_{23}(h)A_1\in R_6$ for small $h$, and
right multiplication by $X_{12}(c_1)$ would again put
$X_{23}(h)A$ in $R_7$ for all small $h$. 
\end{proof}

\begin{lemma}
\label{lem:simple-root-certificate-open}
Let $A>0$. If, for some $i\neq j$, the last-move quadratic
$G_{ij}^A$ has a simple root $t_0$ in $(0,S_{ij})$, then $A$ is an
interior point of
$R_6$ relative to $\mathcal M_+$.
\end{lemma}

\begin{proof}
A simple root varies continuously with the coefficients. Hence, for
all $B$ sufficiently close to $A$, the polynomial $G_{ij}^B$ has a
root $t(B)$ close to $t_0$, with $0<t(B)<S_{ij}(B)$. If $B\in\V$, then
$B\in R_6$. Otherwise
Proposition~\ref{prop:last-move-quadratic-criterion} places $B$ in
$\NH\subseteq R_6$. Thus a neighborhood of $A$ is contained in $R_6$.
\end{proof}

A direct calculation gives
\[
G_{31}^{A_0}(c_6)=0,
\qquad
\bigl(G_{31}^{A_0}\bigr)'(c_6)=-c_3c_4(c_2-c_5),
\]
and
\[
G_{12}^{A_1}(c_7)=0,
\qquad
\bigl(G_{12}^{A_1}\bigr)'(c_7)=-c_4c_5(c_3-c_6).
\]
Undoing the $X_{31}(c_6)$ move on $A_0$ leaves a positive five-move spiral product.
Hence, $c_6$ lies strictly inside the admissible interval. Similarly, $c_7$
lies strictly inside the admissible interval of $G_{12}^{A_1}$. Since
$A_0,A_1\in\partial R_6$, Lemma~\ref{lem:simple-root-certificate-open}
shows that neither root can be simple. Hence
\[
c_5=c_2,
\qquad
c_6=c_3.
\]

The two relations give
\[
A=X_{12}(c_7)X_{31}(c_3)X_{23}(c_2)
  X_{12}(c_4)X_{31}(c_3)X_{23}(c_2)X_{12}(c_1).
\]
A diagonal conjugation now removes the two parameters $c_2$ and $c_3$.
With $D=\operatorname{diag}(c_2c_3,1,c_2)$, the identity
\[
DX_{ij}(c)D^{-1}=X_{ij}\left(\frac{d_i}{d_j}c\right)
\]
sends both occurrences of $c_2$ and both occurrences of $c_3$ to $1$.
Put $x=c_1c_2c_3$, $r=c_4c_2c_3$, and $y=c_7c_2c_3$. The resulting
matrix $DAD^{-1}$ is
\[
U(x,r,y)
=X_{12}(y)X_{31}(1)X_{23}(1)X_{12}(r)X_{31}(1)X_{23}(1)X_{12}(x).
\]
A direct multiplication gives
\begin{equation}\label{eq:critical-family}
U(x,r,y)=
\begin{pmatrix}
1+y & x(1+y)+r+y & r+2y\\
1 & 1+x & 2\\
2 & 2x+r & 1+r
\end{pmatrix}.
\end{equation}
Its entries are positive and $\det U(x,r,y)=1$.

\begin{theorem}[Reduction to the critical family]
\label{thm:critical-family-reduction}
Assume that, for every $x,r,y>0$, there is $\varepsilon>0$ such that
$X_{23}(h)U(x,r,y)\in R_7$ for $0\leq h<\varepsilon$.  Then no pull-in
move  in $\mathcal M_+$ leaves $R_7$. Consequently
every nonsingular attainable $3\times3$ row-stochastic matrix belongs
to $R_7$.
\end{theorem}

\begin{proof}
As argued in the preceding paragraphs,  after relabeling, a hypothetical
first-exit point $A$ of $R_7$ is diagonally equivalent to a critical matrix $U=U(x,r,y)$ and  the outgoing pull-in move is
$X_{23}(h)U$. By hypothesis, however, $X_{23}(h)U\in R_7$ for all
sufficiently small $h\geq 0$, contradicting the assumption that $A$ is a
first-exit point.	
 Thus no finite pull-in path in $\mathcal M_+$ can
leave $R_7$. The finite-product theorem of Johansen--Ramsey then gives the final assertion \cite{JohansenRamsey1979}. 
\end{proof}

\section{The critical region}
\label{sec:critical-region}

In this section we show that a matrix in the critical region cannot be an exit point for $R_7$, and deduce from this the 
main theorem.

Throughout this section, $U=U(x,r,y)$ is the matrix in
\eqref{eq:critical-family}. Put
\[
E=T_{11}(U)=rx-r-3x+1.
\]

If $S$ is the permutation matrix that  interchanges the first two coordinates,
then
\[
U(y,r,x)=SU(x,r,y)^tS.
\]
Hence
\begin{equation}\label{bysymmetry}
U(x,r,y)\in R_n
\quad\Longleftrightarrow\quad
U(y,r,x)\in R_n.
\end{equation}

\subsection{Six-move elimination and the matched
\texorpdfstring{$X_{12}$}{X12}-certificate}

\begin{lemma}
\label{lem:small-r-vestibule}
If $0<r\leq1$, then $U\in\V$.
\end{lemma}

\begin{proof}
We have
\[
V_{13}(U)=-y(r-2)(r+1)-r(r-1)\geq0.
\]
So the vestibule criterion from Theorem \ref{thm:homogeneous-finite-move} applies.
\end{proof}

\begin{lemma}[Six-move elimination]
\label{lem:R6-elimination}
Suppose that $U\notin\V$.  Then $U\in R_6$ in each of the following
regions:
\begin{enumerate}[(1)]
\item $1<r\leq2$ and $\min(x,y)\leq r$;
\item $r\geq3$ and $\max(x,y)\geq r$;
\item $2\leq r\leq3$ and $\min(x,y)\geq r$.
\end{enumerate}
\end{lemma}

\begin{proof}
Since $U\notin\V$, every $V_{ij}(U)$ is negative.

\smallskip
\noindent
\emph{Case (1).}
Assume first that $1<r<2$ and $\min(x,y)<r$. By the symmetry \eqref{bysymmetry}, we may assume that $y<r$. We apply the
last-move quadratic test from
Proposition~\ref{prop:last-move-quadratic-criterion} to
\[
G_{12}^U(t)=V_{21}(X_{12}(-t)U)=\alpha t^2+\beta t+\gamma,
\]
whose coefficients are
\[
\alpha=(x+1)E,
\qquad
\beta=-2(x+1)(yE-x),
\qquad
\gamma=V_{21}(U).
\]
Since 
\[
E=T_{11}(U)=x(r-3)-(r-1)<0,
\]
we have  $\alpha<0$ and $\beta>0$. Since $U\notin \V$, we also have $\gamma=V_{21}(U)<0$.
A direct calculation shows that
\[
\disc(G_{12}^U)
=4x^2(r-2)(x+1)(rx-r-2x).
\]
Since $r-2<0$ and
\[
rx-r-2x=x(r-2)-r<0,
\]
the discriminant $\disc(G_{12}^U)$ is positive.

The three candidates for the admissible endpoint are
\[
\frac{r+2y}{2},
\qquad
y+1,
\qquad
\frac{yx+y+r+x}{x+1}.
\]
Under the present hypotheses, the latter two are strictly larger than the first.
Hence the admissible endpoint is
\[
\bar c=\frac{r+2y}{2}.
\]
The vertex of $G_{12}^U$ is
\[
\tau=\frac{yE-x}{E},
\]
and
\[
\frac{r+2y}{2}-\tau
=\frac{(r-1)(rx-r-2x)}{2E}>0.
\]
Thus $G_{12}^U$ has a root in its admissible interval. By Proposition \ref{prop:last-move-quadratic-criterion}, $U\in R_6$. Proposition~\ref{prop:Rn-relatively-closed} gives the
boundary cases.

\smallskip
\noindent
\emph{Case (2).}
Assume $r>3$ and, by symmetry, $y>r$. We apply the reduced Frydman test
in Proposition~\ref{prop:reduced-Frydman-test} with ordered pair
$(2,1)$. Write
\[
q_{21}^U(t)=\alpha t^2+\beta t+\gamma.
\]
Then
\[
\alpha=(x+1)(yx+y+r+x)E,
\qquad
\beta=-(x+1)\Psi,
\qquad
\gamma=V_{12}(U),
\]
where
\[
\Psi=2V_{12}(U)+(r-2)(-2y+r-4).
\]
Since $V_{12}(U)<0$, and $y>r>3$, we clearly have $\Psi<0$.
Since
\[
V_{12}(U)-E=y(r-3)(x+1)+2r-5>0,
\]
we also have $E<V_{12}(U)<0$. Hence $\alpha<0$, $\beta>0$, and $\gamma<0$.

The complementary minor is $T_{13}(U)=2-r<0$, and
\[
\begin{aligned}
-\widetilde H_{21}(U)={}&\bigl((r-2)(2y+r)^2-8y-8\bigr)x\\
&+(r-2)(2y+r)^2-8y-8r.
\end{aligned}
\]
The quantity
\[
(r-2)(2y+r)^2-8y
\]
increases with $y$ and is larger than
$8r$ at $y=r$.  Thus $\widetilde H_{21}(U)<0$.

The admissible endpoint is
\[
C_{21}=\frac{2}{r+2x}.
\]
The other candidates are
\[
\frac{y+1}{yx+y+r+x}
\qquad\text{and}\qquad
\frac{1}{x+1},
\]
and direct subtraction gives the required ordering. If
$\tau=-\beta/(2\alpha)$, then
\[
\frac2{r+2x}-\tau
=-\frac{(r-2)(2yrx-2yx+2y+r^2+2r+2x)}
{2(r+2x)(yx+y+r+x)E}>0.
\]
The reduced Frydman test gives $U\in R_6$.
Proposition~\ref{prop:Rn-relatively-closed} gives the boundary cases.

\smallskip
\noindent
\emph{Case (3).}
Assume $2<r<3$ and, by symmetry, $r<y\leq x$. We apply the reduced
Frydman test with ordered pair $(3,2)$. For
\[
q_{32}^U(t)=\alpha t^2+\beta t+\gamma,
\]
one has
\[
\alpha=2(r+1)(y(r-3)-r+1),
\qquad
\beta=-(r+1)\Phi,
\qquad
\gamma=V_{23}(U),
\]
where
\[
\Phi=V_{12}(U)-2\bigl((r-2)(x-y)+2(r-2)+2\bigr)<0.
\]
Thus $\alpha<0$, $\beta>0$, and $\gamma<0$.

Put
\[
\Gamma=T_{21}(U)
=yrx-yr-3yx+y-rx+r+x.
\]
Note that $\Gamma$ is affine in each of $x$ and $y$ separately, and that, for fixed \(2<r<3\), we have
\[
\frac{\partial \Gamma}{\partial x}
=yr-3y-r+1
=y(r-3)-r+1<0
\]
and
\[
\frac{\partial \Gamma}{\partial y}
=rx-r-3x+1
=x(r-3)-r+1<0.
\]
Thus \(\Gamma\) is strictly decreasing in
both \(y\) and \(x\). Since \(r<y\leq x\), it follows that
\[
\Gamma<\Gamma|_{y=x=r}
=r^3-5r^2+3r.
\]
Moreover, on the range $2< r \leq 3$ we have $r^3-5r^2+3r<-6$.
Therefore,
$\Gamma<-6$. Consequently, 
\[
T_{21}(U)=\Gamma<0,
\]
and, since \(r+1>3\),
\[
\widetilde H_{32}(U)
=(r+1)\Gamma+16
<-6(r+1)+16<0.
\]

The admissible endpoint is the smaller of
\[
\frac{yx+y+r+x}{2y+r}
\qquad\text{and}\qquad
\frac{x+1}{2}.
\]
The third candidate is larger
than both, since
\[
\frac{r+2x}{r+1}-\frac{yx+y+r+x}{2y+r}
=-\frac{\Gamma}{(2y+r)(r+1)}>0
\]
and
\[
\frac{r+2x}{r+1}-\frac{x+1}{2}
=-\frac{E}{2(r+1)}>0.
\]
For
$\tau=-\beta/(2\alpha)$,
\[
\frac{x+1}{2}-\tau
=\frac{\Gamma+2}{4(y(r-3)-r+1)}>0,
\]
\[
\frac{yx+y+r+x}{2y+r}-\tau
=-\frac{(r-4-2y)\Gamma}
{4(2y+r)(y(r-3)-r+1)}>0.
\]
The reduced Frydman test again gives $U\in R_6$, and
Proposition~\ref{prop:Rn-relatively-closed} gives the boundary cases.
\end{proof}

\begin{lemma}
\label{lem:X12-persistence}
Suppose that $U\notin R_6$ and either one of the following two cases occurs: 
\begin{enumerate}[(1)]
\item	
$1<r<2$ and  $y>r$, 
\item 
$r>2$ and $y<r$.  
\end{enumerate}
Then $X_{23}(h)U\in R_7$ for all sufficiently small $h\geq0$.
\end{lemma}

\begin{proof}
	Put
	\[
	U_h=X_{23}(h)U,
	\qquad
	Q_h(s)=X_{12}(-s)U_h.
	\]
	We show below that for all sufficiently small $h$ there exists $s\geq 0$ such that $Q_h(s)\in R_6$, which in turn shows that $U_h\in R_7$.
	To this end we apply to $Q_h(s)$ the reduced Frydman test of
	Proposition~\ref{prop:reduced-Frydman-test} for the ordered pair
	$(1,2)$.
	
	Set
	\[
	\mathcal N_h(s)=\widetilde H_{12}(Q_h(s)).
	\]

	We first examine the limiting case $h=0$.  Taking $s=y$ removes the
	leftmost factor $X_{12}(y)$ from $U$, giving
	\[
	B:=Q_0(y)
	=
	X_{31}(1)X_{23}(1)X_{12}(r)
	X_{31}(1)X_{23}(1)X_{12}(x).
	\]
	Thus $B\in R_6$.  Moreover, $B\in\NH$. Indeed, $B>0$, so $B\notin\T$. If
	$B\in\V$, then Corollary~\ref{cor:no-exit-vestibule} would give
	$U=X_{12}(y)B\in R_6$, 	contrary to the hypothesis of the lemma. Hence $B\notin\T\cup\V$,
	and therefore $B\in\NH$.
	
	For the particular $(1,2)$-certificate, all four strict inequalities
	in Proposition~\ref{prop:reduced-Frydman-test} hold at $B$, while the
	discriminant factor condition holds with equality.  Indeed,
	\[
	q_{12}^B(t)=-(r-1)(t-x)^2,
	\]
	and
	\[
	C_{12}(B)
	=
	\min\left\{
	r+x,\,1+x,\,x+\frac r2
	\right\}>x.
	\]
	Since $r>1$, we have $T_{22}(B)=1-r<0$. The leading coefficient and linear coefficient of
	$q_{12}^B$ satisfy
	\[
	\alpha_{12}=1-r<0,
	\qquad
	\beta_{12}=2x(r-1)>0.
	\]
	Moreover,
	\[
	2\alpha_{12}C_{12}(B)+\beta_{12}
	=
	-2(r-1)\bigl(C_{12}(B)-x\bigr)<0,
	\]
	and a direct calculation gives
	\[
	T_{23}(B)=-r<0.
	\]
	Finally,
	\[
	\widetilde H_{12}(B)=\mathcal N_0(y)=0.
	\]
	Thus all four strict inequalities in
	Proposition~\ref{prop:reduced-Frydman-test} hold at $B$, while the
	reduced discriminant condition holds with equality.
	
	It remains only to show that, for every sufficiently small $h>0$, one
	can choose $s=s(h)$ with
	\[
	s(h)\to y
	\qquad\text{and}\qquad
	\mathcal N_h(s(h))=0.
	\]
	The polynomial $\mathcal N_h(s)$ is quadratic in $s$.  The explicit
	coefficients are recorded in Appendix~\ref{app:X12-persistence-formulas}.
	In particular,
	\[
	\mathcal N_0(s)=4(r-2)(s-y)^2.
	\]
	Thus $s=y$ is a double root at $h=0$.  If
	\[
	\Delta(h)=\disc_s\bigl(\mathcal N_h(s)\bigr),
	\]
	the formulas in Appendix~\ref{app:X12-persistence-formulas} give
	\[
	\Delta(0)=0,
	\qquad
	\Delta'(0)=64(r-y)(r-2)(r-1).
	\]
	This derivative is positive in both regions of the statement.
	Furthermore, the quadratic coefficient of $\mathcal N_h$ tends to
	$4(r-2)\neq0$.  Hence, for every sufficiently small $h>0$,
	$\mathcal N_h$ has two real roots, and both roots converge to $y$ as
	$h\to0^+$.
	
	Choose either root and call it $s(h)$.  Then
	\[
	s(h)\to y>0,
	\qquad
	Q_h(s(h))\to B>0,
	\]
	so $s(h)>0$ and $Q_h(s(h))>0$ for all sufficiently small $h>0$.
	By construction,
	\[
	\widetilde H_{12}(Q_h(s(h)))=0.
	\]
	The four strict inequalities verified at $B$ persist for
	$Q_h(s(h))$ when $h$ is sufficiently small.
	
	If $Q_h(s(h))\in\V$, then $Q_h(s(h))\in R_6$.  Otherwise
	Proposition~\ref{prop:reduced-Frydman-test} applies and gives
	$Q_h(s(h))\in\NH\subseteq R_6$.  Finally,
	\[
	U_h=X_{12}(s(h))Q_h(s(h)),
	\]
	so $U_h\in R_7$ for all sufficiently small $h>0$.  The case $h=0$
	follows from the  seven-factor representation of $U$.
\end{proof}

\begin{corollary}
\label{cor:remaining-hard-cell}
To prove forward $X_{23}$-persistence for the critical family, it is
enough to treat
\[
2<r<3,
\qquad
0<x<r\leq y,
\qquad
U(x,r,y)\notin R_6.
\]\end{corollary}

\begin{proof}
Assume that $U\notin R_6$. Lemma~\ref{lem:small-r-vestibule} gives
$r>1$. If $1<r<2$, Lemma~\ref{lem:R6-elimination}(1) forces $y>r$, so
Lemma~\ref{lem:X12-persistence} applies. If $r=2$, parts (1) and (3) of
Lemma~\ref{lem:R6-elimination} cover all parameters. If $2<r<3$, part
(3) shows that either $y<r$, in which case
Lemma~\ref{lem:X12-persistence} applies, or $x<r\leq y$, which is the displayed cell. 
Finally, if $r\geq3$,
part (2) forces $y<r$, and Lemma~\ref{lem:X12-persistence} applies.
\end{proof}

\subsection{\texorpdfstring{The $X_{23}$-certificate}{The X23-certificate}}
\label{subsec:X23-certificate}

By Corollary~\ref{cor:remaining-hard-cell}, it remains to consider
\[
2<r<3,
\qquad
0<x<r\leq y,
\qquad
U:=U(x,r,y)\notin R_6.
\]
These are the standing hypotheses throughout this subsection.

Put
\[
E=T_{11}(U)=rx-r-3x+1<0.
\]
We shall also use
\[
T_{21}(U)=yrx-yr-3yx+y-rx+r+x<0.
\]
Indeed, on writing $u=r-2$ and $w=y-r$, one has
\[
T_{21}(U)
=x\bigl(u^2+uw-w-3\bigr)
 -\bigl(u^2+uw+2u+w\bigr)<0.
\]

For $g\leq0$, put
\[
U_g=X_{23}(g)U.
\]
Thus $g=0$ gives $U$, while $g<0$ subtracts a positive
multiple of the third row from the second.  Explicitly,
\[
U_g=
\begin{pmatrix}
1+y & x(1+y)+r+y & r+2y\\
1+2g & 1+x+(r+2x)g & 2+(r+1)g\\
2 & r+2x & r+1
\end{pmatrix}.
\]

\begin{lemma}
\label{lem:outer-admissible-interval}
The matrix $U_g$ is positive precisely for
\[
g_{\min}<g\leq0,
\qquad
g_{\min}=-\frac{x+1}{r+2x}.
\]
Moreover, no point $U_g$, $g_{\min}<g<0$, belongs to the
vestibule.
\end{lemma}

\begin{proof}
The first and third rows of $U_g$ are fixed and positive.  Since $r>2$, the
second coordinate of the second  row of $U_g$ vanishes before either of the
other two:
\[
-\frac{x+1}{r+2x}>-\frac12,
\qquad
-\frac{x+1}{r+2x}>-\frac{2}{r+1},
\]
where the second inequality is equivalent to $E<0$.  This proves
the stated admissible interval.

If $U_g\in\V$ for some $g\in(g_{\min},0)$, then
$U=X_{23}(-g)U_g$ would belong to $R_6$ by Corollary~\ref{cor:no-exit-vestibule}, contrary to the
standing hypotheses.
\end{proof}

For $g\in(g_{\min},0)$, consider Frydman's quadratic for the
ordered pair $(2,3)$ at $U_g$:
\[
q_g(t):=q_{23}^{U_g}(t)
=\alpha(g)t^2+\beta(g)t+\gamma(g).
\]
Set
\[
m(g)=x+1+(r+2x)g>0.
\]

The following identities are direct applications of the formulas in
Section~\ref{sec:preliminaries}.

\begin{lemma}
	\label{lem:X23-reduced-data}
	For $g\in(g_{\min},0)$, the following hold:
\begin{enumerate}[(i)]	
\item
The admissible endpoint for $q_g$ is
	\[
	\bar c=\frac{r+1}{r+2x}.
	\]

\item
The coefficients of $q_g$ satisfy 
	\[
	\alpha(g)=(r+2x)m(g)T_{33}(U_g),
	\qquad
	\beta(g)=-2m(g)B(g),
	\qquad
	\gamma(g)=V_{32}(U_g),
	\]
	where
	\[
	\begin{aligned}
		B(g)={}&-(r^2+rx-r-2x)\\
		&+g\bigl(yr^2+yrx-2yr-3yx-y-r^2-rx+x\bigr).
	\end{aligned}
	\]
	
\item	
	The vertex-location expression satisfies
	\[
	2\alpha(g)\bar c+\beta(g)
	=2m(g)\bigl(T_{31}(U_g)+1\bigr).
	\]

\item	
The reduced discriminant factor from
	Lemma~\ref{lem:Frydman-discriminant-factorization} is
	\[
	\widetilde H_{23}(U_g)=4\Theta(g),
	\]
	where
	\[
	\Theta(g)=h_2g^2+h_1g+x^2(r-2).
	\]
	Here
	\[
	h_2=-(r+2x)T_{21}(U)>0,
	\]
	and the explicit formula for $h_1$, together with the remaining
	polynomial identities used below, is recorded in
	Appendix~\ref{app:X23-Frydman-formulas}.
\end{enumerate}	
\end{lemma}

\begin{proof}
	The three candidates for the admissible endpoint are
	\[
	\frac{r+2y}{x(1+y)+r+y},
	\qquad
	\frac{2+(r+1)g}{x+1+(r+2x)g},
	\qquad
	\frac{r+1}{r+2x}.
	\]
	The last of these is the smallest (under the standing hypotheses). Indeed,
	\[
	\frac{r+2y}{x(1+y)+r+y}
	-\frac{r+1}{r+2x}
	=
	\frac{-T_{21}(U)}
	{(r+2x)\bigl(x(1+y)+r+y\bigr)}>0,
	\]
	while
	\[
	\frac{2+(r+1)g}{x+1+(r+2x)g}
	-\frac{r+1}{r+2x}
	=
	\frac{-E}
	{(r+2x)\bigl(x+1+(r+2x)g\bigr)}>0.
	\]
	Hence
	\[
	\bar c=\frac{r+1}{r+2x}.
	\]
	
	The remaining identities follow by direct substitution into the
	formulas of Section~\ref{sec:preliminaries}.
\end{proof}

Let us show first that the leading-coefficient condition from Proposition \ref{prop:reduced-Frydman-test}
applied to the Frydman quadratic $q_g(t)$ holds throughout the 
admissible interval $(g_{\min},0)$.

\begin{lemma}
\label{lem:X23-concavity}
Under the standing hypotheses, for every $g\in(g_{\min},0)$, one has
\[
T_{33}(U_g)<0,
\qquad
\alpha(g)<0.
\]
\end{lemma}

\begin{proof}
The modified minor $T_{33}(U_g)$ is affine in $g$.  At the two endpoints,
\[
T_{33}(U)=1-r<0
\]
and
\[
T_{33}(U_{g_{\min}})
=-\frac{(r-2)(yx+y+r+x)}{r+2x}<0.
\]
Hence it is negative throughout the interval.  The formula for
$\alpha(g)$ and the positivity of $m(g)$ give the conclusion.
\end{proof}

Observe that the reduced discriminant factor $\Theta(g)$ is a quadratic in $g$.
We let
\[
g_*=-\frac{h_1}{2h_2}
\]
denote the vertex of the quadratic $\Theta(g)$.

In the following proposition we establish that $g_*$ lies in the admissible $(g_{\min},0)$, that the discriminant of $q_{g_*}(t)$ is positive, and that the vertex-location condition for $q_{g_*}(t)$ has the required sign.

\begin{proposition}
\label{prop:X23-automatic-signs}
Under the standing hypotheses,
\[
g_{\min}<g_*<0,
\qquad
\Theta(g_*)<0,
\qquad
T_{31}(U_{g_*})+1<0.
\]
\end{proposition}

\begin{proof}
Write $u=r-2$ and $w=y-r$, so that
$0<u<1$ and $w\geq 0$.  The coefficient $h_1$ takes the form
\[
\begin{aligned}
h_1={}&x\Bigl((3+2u-u^2)x+(u^2+2u+3)\Bigr)\\
&+w\Bigl((1-u)x^2+2x+u+1\Bigr)>0.
\end{aligned}
\]
Since $h_2=-(r+2x)T_{21}(U)>0$, it follows that $g_*<0$.

Put
\[
N_{\min}=2h_2(x+1)-h_1(r+2x).
\]
Then
\[
g_*-g_{\min}
=\frac{N_{\min}}{2h_2(r+2x)}.
\]
A direct simplification gives
\[
\begin{aligned}
N_{\min}=-(r+2x)\Bigl(&
(u-1)(u+3)x^2+(u-3)(u+1)x\\
&-2u(u+2)+w(x+1)E\Bigr).
\end{aligned}
\]
The first three terms inside the parentheses are strictly negative,
while $w(x+1)E\leq0$. Hence $N_{\min}>0$. Thus $g_*>g_{\min}$.

The discriminant of $\Theta$, viewed as a quadratic in $g$, is
\[
\disc_g(\Theta)=EH_+H_-,
\]
where
\begin{align*}
H_+ &=y(x+1)^2+rx-r+x^2-3x,\\
H_- &=x(y+1)(r-3)-(r-1)(y-r).
\end{align*}
Here $H_-<0$, as $2< r<3$ and $r\leq y$, while
\[
H_+\geq r(x+1)^2+rx-r+x^2-3x
=(r+1)x^2+3(r-1)x>0.
\]
Since $E<0$, the discriminant is positive.  The quadratic
$\Theta$ is convex and $g_*$ is its vertex, so
$\Theta(g_*)<0$.

Direct calculation shows that
\[
2h_2\bigl(T_{31}(U_{g_*})+1\bigr)
=-(yx+y+r+x)E\,T_{21}(U).
\]
Both $E$ and $T_{21}(U)$ are negative, so the right-hand side is
negative.  Hence $T_{31}(U_{g_*})+1<0$.
\end{proof}

Next we handle the sign of the linear term of the Frydman quadratic $q_{g_*}(t)$.

\begin{proposition}
\label{prop:B-tangency-order}
One has $\beta(g_*)>0$.
\end{proposition}

\begin{proof}
Since $\beta(g)=-2m(g)B(g)$ and $m(g)>0$ for $g\in (g_{\min},0)$,
it will suffice to show that $B(g_*)<0$.

Write $B(g)=b_0+b_1g$, where
\[
b_0=-(r^2+rx-r-2x)<0.
\]
Suppose that $b_1\geq0$. Then $g_*<0$ gives
$B(g_*)\leq B(0)<0$, and we are done.  

Suppose that $b_1<0$. Let $g_B=-b_0/b_1<0$ be  the zero of $B(g)$. It will suffice to show that $g_B<g_*$.

Let
\[
\lambda=yr+yx-y+x>0
\]
and define
\[
\mathcal H(g)=(2y+r)\Theta(g)+\lambda\gamma(g).
\]
Exact polynomial division gives
\[
\mathcal H(g)=B(g)G(g),
\]
where
\[
G(g)=y(r-1)(x+1)+x
 +(r+2x)\bigl(y(r-1)+1\bigr)g.
\]
The slope of $G$ is positive, whereas the slope of $B(g)=b_0+b_1g$ is
negative.  Hence $g\mapsto \mathcal H(g)$ is a concave quadratic with two real roots, one of which is $g_B$.

Since no matrix $U_g$ with $g\in(g_{\min},0)$ belongs to $\V$,
$\gamma(g_*)=V_{32}(U_{g_*})<0$.  Together with
$\Theta(g_*)<0$, this gives $\mathcal H(g_*)<0$.  Moreover, 
differentiating with respect to $g$ in 
the definition of $\mathcal H(g)$ and setting $g=g_*$ we get
\[
\mathcal H'(g_*)=\lambda\gamma'(g_*),
\]
where we have used that  $g_*$ is the vertex of $\Theta$.  Direct calculation shows that
\[
h_2\gamma'(g_*)=-(r+1)(r+2x)L,
\]
where $h_2$ is as in Lemma \ref{lem:X23-reduced-data} and
\[
\begin{aligned}
	L={}&2u^3+2u^2w+3u^2x+8u^2+3uwx+5uw\\
	&+6ux+8u+wx+w+3x.
\end{aligned}
\]
Every term in $L$ is nonnegative, and the term $3x$ is strictly
positive. Therefore $L>0$. Since $h_2>0$, $\gamma'(g_*)<0$. This shows that 
$\mathcal H'(g_*)<0$.

We have shown that $\mathcal H(g)$ is a concave quadratic
that is negative and decreasing at $g_*$. It follows that $g_*$ lies to the right of both
of its roots.  Thus $g_B<g_*$, and the negative slope of $B$
then gives $B(g_*)<0$.  Finally,
$\beta(g_*)=-2m(g_*)B(g_*)>0$.
\end{proof}

\begin{theorem}
\label{thm:hard-cell-predecessor}
Under the standing hypotheses,
\[
g_*\in(g_{\min},0)
\qquad\text{and}\qquad
U_{g_*}\in\NH\subseteq R_6.
\]
\end{theorem}

\begin{proof}
Lemma~\ref{lem:outer-admissible-interval} gives
$U_{g_*}\notin\V$.  The remaining hypotheses of the reduced
Frydman test are
\[
T_{33}(U_{g_*})<0,
\qquad
\beta(g_*)>0,
\qquad
2\alpha(g_*)\bar c+\beta(g_*)<0,
\qquad
\widetilde H_{23}(U_{g_*})<0.
\]
They follow from Lemma~\ref{lem:X23-concavity} and
Proposition~\ref{prop:X23-automatic-signs}, and
Proposition~\ref{prop:B-tangency-order}.  Hence $q_{g_*}$ has an admissible root and $U_{g_*}\in\NH$.
\end{proof}

\begin{corollary}
\label{cor:X23-forward-persistence}
Under the hypotheses of Theorem~\ref{thm:hard-cell-predecessor},
$X_{23}(h)U\in R_7$ for all $h\geq 0$.
\end{corollary}

\begin{proof}
Since $U=X_{23}(-g_*)U_{g_*}$, one has
\[
X_{23}(h)U=X_{23}(h-g_*)U_{g_*}.
\]
Here $h-g_*>0$, and $U_{g_*}\in R_6$, so the right-hand side is
a seven-move representation.  Row normalization gives the corresponding
stochastic pull-in path.
\end{proof}

\begin{proof}[Proof of Theorem~\ref{thm:main}]
By Theorem~\ref{thm:critical-family-reduction}, it suffices
to show that for every critical matrix $U=U(x,r,y)$ one has $X_{23}(h)U\in R_7$ for all small enough $h>0$.

If $U\in R_6$, this is immediate. Assume that $U\notin R_6$.
Lemma~\ref{lem:small-r-vestibule} gives $r>1$. 

If $1<r<2$, then
Lemma~\ref{lem:R6-elimination} (1) forces $y>r$, and
Lemma~\ref{lem:X12-persistence} applies. 

If $r=2$, parts (1) and (3)
of Lemma~\ref{lem:R6-elimination} cover all parameters. 

If $2<r<3$, then either $y<r$, in which case
Lemma~\ref{lem:X12-persistence} applies, or $x<r\leq y$, in which case
Corollary~\ref{cor:X23-forward-persistence} applies.

 If $r\geq3$,
Lemma~\ref{lem:R6-elimination} (2) forces $y<r$, and
Lemma~\ref{lem:X12-persistence} applies. 

This shows that every critical matrix has
the required persistence in $R_7$, and Theorem~\ref{thm:critical-family-reduction}
concludes the proof.
\end{proof}

\section*{Statement on the use of artificial intelligence}

This paper was developed through an extended mathematical collaboration
with OpenAI's ChatGPT, using the GPT-5.6 Sol model and Codex through a ChatGPT Plus account. I supplied ChatGPT with an initial report on the problem which discussed  relevant ideas and identities and the overall strategy of studying a first exit from
$R_7$. Starting from this framework, ChatGPT assisted in solving the  algebraic conditions arising at first exit and in finding a convenient normalization of the matrices satisfying them. This led ChatGPT to the explicit three-parameter critical family \(U(x,r,y)\) and to more tractable formulas for the discriminants and their derivatives. Codex was used for numerical exploration of the critical family and to identify
which certificate appeared to remain active in the unresolved
parameter region. Guided by that exploration, ChatGPT proved the persistent \(X_{23}\)-certificate. The proof then underwent several rounds of revision in collaboration with ChatGPT, during which the original argument was progressively reorganized and simplified into the form presented here.
I take full responsibility for the contents of the paper and for any errors 
that remain.

\appendix
\section{The quadratic used in the
	\texorpdfstring{$X_{12}$}{X12}-persistence argument}
\label{app:X12-persistence-formulas}
This appendix records explicit formulas for the coefficients and the discriminant of the quadratic $\mathcal N_h(s)$
that appears in Lemma \ref{lem:X12-persistence}.
 
Recall that in the proof of Lemma \ref{lem:X12-persistence}, we set
$U=U(x,r,y)$,
\[
U_h=X_{23}(h)U,
\qquad
Q_h(s)=X_{12}(-s)U_h.
\]

\subsection*{Explicit entries of $Q_h(s)$}
Direct multiplication gives
\small{\[
Q_h(s)=
\begin{pmatrix}
	1+y-(1+2h)s
	&
	x(1+y)+r+y-s\bigl(1+x+h(r+2x)\bigr)
	&
	r+2y-s\bigl(2+h(r+1)\bigr)
	\\[2mm]
	1+2h
	&
	1+x+h(r+2x)
	&
	2+h(r+1)
	\\
	2&r+2x&r+1
\end{pmatrix}.
\]}

\subsection*{The quadratic $\mathcal N_h(s)$}
Recall that  $\mathcal N_h(s)=\widetilde H_{12}(Q_h(s))$.
A direct calculation shows that
\[
\mathcal N_h(s)
=A_2(h)(s-y)^2+A_1(h)(s-y)+A_0(h),
\]
where
\begin{align*}
A_2(h) &=(r-2)(1+2h)\bigl(2+(r+1)h\bigr)^2,\\
A_1(h) &= 2h\bigl(2+(r+1)h\bigr)
\left[
2y(r-2)+3(r-1)
+h(r+1)\bigl(y(r-2)+r\bigr)
\right],
\end{align*}
and
\[
\begin{aligned}
	A_0(h)
	={}&h\Bigl[
	4(r-1)(y-r)\\
	&\qquad
	+h(r+1)\bigl(8y(r-1)-r(r+1)\bigr)
	+2yr(r+1)^2h^2
	\Bigr].
\end{aligned}
\]

In particular,
\[
A_2(0)=4(r-2),
\qquad
A_1(0)=A_0(0)=0,
\]
and hence
\[
\mathcal N_0(s)=4(r-2)(s-y)^2.
\]

\subsection*{The discriminant of $\mathcal N_h(s)$}
Let
\[
\Delta(h)=\disc_s(\mathcal N_h(s))
=A_1(h)^2-4A_2(h)A_0(h).
\]
Using that $A_0'(0)=4(r-1)(y-r)$,
we obtain
\[
\begin{aligned}
	\Delta'(0)
	&=-4A_2(0)A_0'(0)\\
	&=64(r-y)(r-2)(r-1).
\end{aligned}
\]

\section{Explicit formulas for the \texorpdfstring{$X_{23}$}{X23}-Frydman quadratic}
\label{app:X23-Frydman-formulas}

Let
\[
U=U(x,r,y),\qquad U_g=X_{23}(g)U.
\]
This appendix records the polynomial identities used in
Section~\ref{subsec:X23-certificate} in the application of the reduced Frydman test (Proposition \ref{prop:reduced-Frydman-test}) to $U_g$.

\subsection*{Modified minors}
The modified minors needed in the argument are
\begin{align*}
T_{11}(U_g) &=E=rx-r-3x+1,\\
T_{21}(U_g) &=yrx-yr-3yx+y-rx+r+x,\\
T_{33}(U_g) &=1-r+\bigl(y(r-2)-r\bigr)g,
\end{align*}
and
\[
\begin{aligned}
T_{31}(U_g)={}&rx-r-2x\\
&+g\bigl(-yrx+yr+3yx-y+rx-r-x\bigr).
\end{aligned}
\]

\subsection*{Frydman quadratic coefficients}
For the ordered pair $(2,3)$, write
\[
q_g(t):=q_{23}^{U_g}(t)=\alpha(g)t^2+\beta(g)t+\gamma(g).
\]
Then, with
\[
m(g)=x+1+(r+2x)g,
\]
one has
\begin{align*}
\alpha(g) &=(r+2x)m(g)T_{33}(U_g),\\
\beta(g) &=-2m(g)B(g),
\end{align*}
where
\[
\begin{aligned}
B(g)={}&-(r^2+rx-r-2x)\\
&+g\bigl(yr^2+yrx-2yr-3yx-y-r^2-rx+x\bigr).
\end{aligned}
\]
The constant term is
\[
\gamma(g)=V_{32}(U_g)=c_2g^2+c_1g+c_0,
\]
where
\begin{align*}
c_2 &=(r+1)(r+2x)(yr-3y-r+1),\\
c_1 &=(r+1)\bigl(yrx+yr-3yx-3y-r^2-3rx+r+5x\bigr),\\
c_0 &=-r^2x-r^2+rx+r+2x.
\end{align*}

\subsection*{Admissible endpoint}
The admissible endpoint of the Frydman quadratic $q_g(t)$ is
\[
\bar c=\frac{r+1}{r+2x}.
\]

\subsection*{Vertex location}
The vertex-location expression for $q_g(t)$ simplifies to
\[
2\alpha(g)\bar c+\beta(g)
=2m(g)\bigl(T_{31}(U_g)+1\bigr).
\]

\subsection*{Reduced discriminant factor $\widetilde H_{23}$}
The reduced discriminant factor $\widetilde H_{23}$ of $U_g$ is
\[
\widetilde H_{23}(U_g)=4\Theta(g),
\qquad
\Theta(g)=h_2g^2+h_1g+x^2(r-2),
\]
where
\[
\begin{aligned}
h_1={}&-yrx^2+yr+3yx^2+2yx-y\\
&+r^2x-r^2+3rx^2-4rx+r-5x^2+3x
\end{aligned}
\]
and
\[
h_2=-(r+2x)\bigl(yrx-yr-3yx+y-rx+r+x\bigr).
\]

We also have the identity 
\[
(2y+r)\Theta(g)+\lambda\gamma(g)=B(g)G(g),
\]
where $B(g)$ has been defined above and 
\begin{align*}
\lambda &=yr+yx-y+x,\\
G(g) &=y(r-1)(x+1)+x
+(r+2x)\bigl(y(r-1)+1\bigr)g.
\end{align*}

\subsection*{The discriminant of $\Theta(g)$}
The discriminant of the quadratic $g\mapsto \Theta(g)$ factors as
\[
\disc_g(\Theta)=EH_+H_-,
\]
where
\begin{align*}
H_+ &=y(x+1)^2+rx-r+x^2-3x,\\
H_- &=x(y+1)(r-3)-(r-1)(y-r).
\end{align*}

\subsection*{Identities at $g=g_*$}
At the vertex $g_*=-h_1/(2h_2)$, the following identities hold:
\[
2h_2\bigl(T_{31}(U_{g_*})+1\bigr)
=-(yx+y+r+x)E\,T_{21}(U),
\]
and
\[
h_2\gamma'(g_*)
=-(r+1)(r+2x)
\bigl(2yr^2+3yrx-3yr-5yx-y-r^2-rx+r+3x\bigr).
\]

\end{document}